\documentclass[12pt]{article} 
\usepackage{cmap}
\usepackage{ucs}
\usepackage[utf8]{inputenc}
\usepackage[top=2.5cm,bottom=2.5cm,left=2.5cm,right=2.5cm]{geometry}
\usepackage{amsfonts,amsmath,amsthm,amscd,amssymb,latexsym}
\usepackage{mathtools}
\usepackage[english]{babel}
\usepackage{tikz}
\usepackage{subcaption}
\usetikzlibrary{positioning}
\usetikzlibrary{intersections}
\usetikzlibrary{arrows.meta, decorations.markings}
\usetikzlibrary{shapes.geometric, arrows}
\usepackage{float}
\usepackage{csquotes}
\usepackage{enumitem}
\usepackage{comment}
\usepackage{adjustbox}
\usepackage[labelfont=bf, labelsep=period]{caption}
\usepackage[colorinlistoftodos]{todonotes}
\usepackage[output-decimal-marker={,}]{siunitx}
\usepackage{listings}
\usepackage{xcolor}
\usepackage{array}
\usepackage{anyfontsize}
\usepackage{ifthen}

\usepackage{authblk}

\usepackage[pdftex,colorlinks,linkcolor=blue,citecolor=red,urlcolor=blue,hyperindex,plainpages=false,bookmarksopen,bookmarksnumbered,unicode]{hyperref}

\definecolor {processblue}{cmyk}{0.96,0,0,0}
\definecolor{codegreen}{rgb}{0,0.6,0}
\definecolor{codegray}{rgb}{0.5,0.5,0.5}
\definecolor{codepurple}{rgb}{0.58,0,0.82}
\definecolor{backcolour}{rgb}{0.95,0.95,0.92}
\definecolor{darkgreen}{RGB}{1,185,32}

\makeatletter 
\newcount\SOUL@minus
\makeatother  

\lstdefinestyle{mystyle}{
	backgroundcolor=\color{backcolour},   
	commentstyle=\color{codegreen},
	keywordstyle=\color{blue},
	numberstyle=\tiny\color{codegray},
	stringstyle=\color{codepurple},
	basicstyle=\ttfamily\footnotesize,
	breakatwhitespace=false,         
	breaklines=true,                 
	captionpos=b,                    
	keepspaces=true,                 
	numbers=left,                    
	numbersep=5pt,                  
	showspaces=false,                
	showstringspaces=false,
	showtabs=false,                  
	tabsize=2
}

\newcommand{\ceil}[1]{\left\lceil{#1}\right\rceil}

\newtheorem{theorem}{Theorem}[section]
\newtheorem{lemma}[theorem]{Lemma}
\newtheorem{proposition}[theorem]{Proposition}
\newtheorem{corollary}[theorem]{Corollary}
\newtheorem{conjecture}[theorem]{Conjecture}

\theoremstyle{definition}
\newtheorem{definition}[theorem]{Definition}

\theoremstyle{remark}

\begin{document}
\large
\title{\vspace{-2cm} Regular $K_3$-irregular graphs}
	
\author[1]{Artem Hak \thanks{Corresponding author: artikgak@ukr.net}}
\author[2]{Sergiy Kozerenko}
\author[3]{Andrii Serdiuk}

\affil[1]{\footnotesize National University of Kyiv-Mohyla Academy, 2 Skovorody Str., 04070 Kyiv, Ukraine}

\affil[1,2,3]{\footnotesize Kyiv School of Economics, 3 Mykoly Shpaka Str., 03113 Kyiv, Ukraine}

\affil[3]{\footnotesize McGill University, Montreal, Quebec H3A 0G4, Canada}

	\date{\vspace{-5ex}}
	
	\maketitle
	
	\begin{abstract}
		We address the problem proposed by Chartrand, Erd\H{o}s and Oellermann (1988) about the existence of regular $K_3$-irregular graphs. We first establish bounds on the $K_3$-degrees of such graphs and use them to prove that there are no such graphs with regularities at most $7$. For the regularity $8$, we narrow down the bounds on the order of such graphs to six possible values. We then present an explicit example of a $9$-regular $K_3$-irregular graph. Finally, we discuss an evolutionary algorithm developed to discover such graphs. Using it, we have found such graphs for consecutive regularities from $9$ to $30$.
	\end{abstract}
	{ \small
	}

\section{Introduction}

A graph is called regular if all its vertices have the same degree. In contrast, constructing a graph in which all degrees are distinct is impossible (except for the trivial case of a one-vertex graph). Nevertheless, many alternative notions have been introduced to measure graph irregularity; see, for example, the overview in~\cite{Chartrand:2016}. One such approach uses the notion of the $F$-degree introduced by Chartrand, Holbert, Oellermann, and Swart~\cite{Char:87}. This concept generalizes the classical vertex degree: given two graphs $G$ and $F$, the $F$-degree of a vertex $v$ in $G$ is defined as the number of subgraphs of $G$ that are isomorphic to $F$ and contain $v$. A graph $G$ is said to be $F$-irregular if all its vertices have distinct $F$-degrees.

In the seminal paper~\cite{Char-Erd-Oell:88}, Chartrand, Erd\H os, and Oellermann posed the problem of whether regular $K_3$-irregular graphs exist. The difficulty of the problem lies in the tension between global uniformity and local asymmetry: every vertex must have the same degree, but participate in a different number of triangles. This problem remained open for decades.

The study of triangle-degrees continued in works of Nair and Vijayakumar, where they investigated the relation between the triangle-degrees of a vertex in a graph and its complement~\cite{Nair:94} and studied edge triangle-degrees~\cite{Nair:96}. 
More recently, Berikkyzy et al.~\cite{dist-triangle:2024} presented the smallest possible $K_3$-irregular graph, gave bounds on triangle-degrees of graphs and restated the question if regular triangle-irregular graphs exist.

In 2024, the first regular $K_3$-irregular graphs were discovered by Stevanovi\'c et al.~\cite{reg-triangle:24}. They found examples of regular $K_3$-irregular graphs for regularities $r \in \{10,11,12\}$ using a mix of heuristic search and structural insights. However, smaller regularities remained elusive.

In this paper, we make several contributions to this line of research. We prove that regular $K_3$-irregular graphs do not exist for regularities $r \leq 7$. Next, we explore the critical case of $r=8$ and establish lower and upper bounds on the order of these graphs, narrowing it down to six possible values. Finally, we present the first known example of a $9$-regular $K_3$-irregular graph (see Figure~\ref{9-reg-k3-irreg-figure} and Table~\ref{tab:9reg-k3-degrees-neighbors}). Using an evolutionary algorithm, we discovered examples of regular $K_3$-irregular graphs for every regularity from $9$ up to $30$.

While evolutionary algorithms have a long history in optimization and applied mathematics~\cite{genetic-survey}, their application to constructive problems in pure mathematics, and combinatorics in particular, remains relatively rare. Only a handful of studies have explored the use of such techniques to support or refute mathematical conjectures.
For example, Miasnikov~\cite{Miasnikov:99} utilized genetic algorithms to investigate the Andrews--Curtis conjecture on the balanced co-representation of trivial groups. He proved that it holds for the potential counterexamples (known at that time). More recently, Wagner~\cite{AdamWagner:21} applied reinforcement learning to discover counterexamples to several conjectures in spectral graph theory and extremal combinatorics. This line of research continues: Wagner and collaborators~\cite{wagner2025multiagent} introduced a multi-agent AI system capable of constructing complex geometric structures such as polytopes --- further demonstrating the potential of such methods in pure mathematics.

The paper is organized as follows. Section~\ref{sect-defs-prelim} introduces the necessary definitions and preliminary results.
Section~\ref{sect-bounds} develops a general technique used in subsequent proofs and establishes both the lower and upper bounds on the graph parameters.
Section~\ref{sect-small-regs} provides an analytical proof that no $r$-regular $K_3$-irregular graphs exist for $r \leq 7$.
Section~\ref{sect-8-reg} investigates the case $r=8$ and derives bounds on the possible order of such graphs.
In Section~\ref{sect-9-reg-plus}, we present an example of a $9$-regular $K_3$-irregular graph and discuss its structural properties.
Section~\ref{sect-algorithm} discusses the implementation details of the evolutionary search algorithm.

We note that some results of this paper were announced at the Ukraine Mathematics Conference “At the End of the Year 2024”~\cite{HakEndOfYear:24}.
 
\section{Main definitions and preliminary results}\label{sect-defs-prelim}

In this section, we introduce the key concepts related to $F$-degrees and $F$-irregular graphs, which will be used throughout the paper. We also state preliminary results on the properties of $K_3$-degrees and their behavior in the graph complement.

\subsection{Main definitions}

A graph is an ordered pair $G = (V, E)$ where $V = V(G)$ is the set of its \textit{vertices} and $E = E(G) \subset \binom{V}{2}$ is the set of its \textit{edges}. All the graphs considered in this paper are simple and finite. Also, for a pair of vertices $u, v \in V(G)$, the edge $\{u, v\}$ will be denoted by $uv$. By $\overline{G}$, we denote the \textit{complement} of a graph $G$. We use $K_n$ and $P_n$ to denote the complete graph and the path on $n$ vertices, respectively. 

For two sets of vertices $A, B \subset V(G)$, by $E(A, B)$ we denote the set of edges between $A$ and $B$. 

The \textit{neighborhood} of a vertex $v$ in a graph $G$ is the set of all its adjacent vertices: $N_G(v) = \{u \in V(G) \mid uv \in E(G)\}$. The \textit{closed neighborhood} of $v$ is the set $N_G[v] = N_G(v) \cup \{v\}$. The \textit{degree} of $v$ in $G$ is the number $\deg_G(v) = |N_G(v)|$. A vertex $v \in V(G)$ is an \textit{isolated vertex} provided $\deg_G(v) = 0$. The graph $G$ is said to be \textit{regular} if all its vertices have the same degree, which is called the \textit{regularity} of $G$. 

Two vertices $u$ and $v$ are said to be \textit{false twins} if $N_G(u) = N_G(v)$ and $uv \notin E(G)$. If instead $N_G[u] = N_G[v]$, then $u$ and $v$ are called \textit{twins}, or \textit{true twins}.

Two graphs $G$ and $H$ are called \textit{isomorphic} if there is an \textit{isomorphism} between them, that is, a bijection $f: V(G) \rightarrow V(H)$ such that $uv \in E(G)$ if and only if $f(u)f(v) \in E(G)$.

For a set of vertices $A \subset V(G)$, by $G[A]$ we denote the subgraph of $G$ induced by $A$. We also put $E(A)=E(G[A])$. 

Let $u_1$, $v_1$ and $u_2$, $v_2$ be four distinct vertices in $G$ with $u_1v_1$, $u_2v_2$ $\in E(G)$ and $u_1u_2$, $v_1v_2$ $\notin E(G)$. The \textit{$2$-switch} operation on these edges produces a graph obtained from $G$ by removing the edges $u_1v_1$, $u_2v_2$ and adding new edges $u_1u_2$, $v_1v_2$.

\subsection{F-irregular graphs}

\begin{definition}[\cite{Char-Erd-Oell:88, Char:87}]
For a given graph $F$, the \textit{$F$-degree} of a vertex $v$ in~$G$ is the number $F \deg(v)$ of subgraphs of $G$, isomorphic to $F$, to which $v$ belongs.
\end{definition}

Note that the ordinary degree of a vertex is exactly its $K_2$-degree.

\begin{definition}[\cite{Char-Erd-Oell:88, Char:87}]
A graph $G$ is called \textit{$F$-irregular} if all its vertices have distinct $F$-degrees.
\end{definition}

Figure~\ref{fig-1} illustrates the smallest possible $K_3$-irregular graph. It has $7$ vertices, $15$ edges, degree sequence $(6, 5, 5, 4, 4, 3, 3)$, and the corresponding $K_3$-degree sequence $(9, 7, 6, 5, 4, 3, 2)$. The graph was found by the computer search in~\cite{dist-triangle:2024}.

\begin{figure}
    \centering
    \begin{tikzpicture}[auto,on grid, state/.style ={circle, top color =black , bottom color = black , draw, minimum width=1mm}, inner sep=1.5pt, scale = 1]

	\node[state, label=left:$v_4$] (v4) {};
	\node[state, label=below right:$v_5$] (v5) [below = 15mm of v4] {};
	\node[state, label=above right:$v_2$] (v2) [right = 15mm of v5] {};
	\node[state, label=above right:$v_7$] (v7) [below = 15mm of v5] {};
    \node[state, label=below:$v_3$] (v3) [right = 15mm of v7] {};
    \node[state, label=below left:$v_1$] (v1) [below left = 15mm of v7] {};
    \node[state, label=below right:$v_6$] (v6) [below right = 15mm of v3] {};
    
    \draw (v1) .. controls (-2, 2) and (1.5, 2) .. (v2);

    \foreach \i in {3,...,7}
        \draw (v1)--(v\i);

    \foreach \i in {3,...,6}
        \draw (v2)--(v\i);

	\draw (v3) -- (v4);
    \draw (v3) -- (v6);
    \draw (v3) -- (v7);

    \draw (v4) -- (v5);

    \draw (v5) -- (v7);
			
    \end{tikzpicture}
    \caption{The smallest $K_3$-irregular graph~\cite{dist-triangle:2024}.}\label{fig-1}
\end{figure}
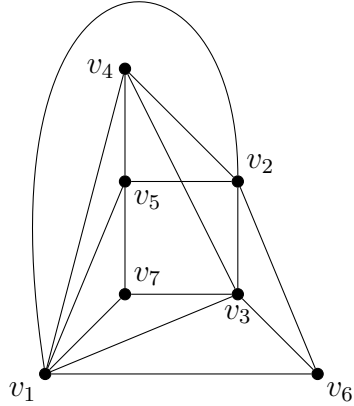

Many of our proofs involve analyzing the relationship between \textit{triangle-degrees} (\( K_3 \)-degrees) in a graph and its complement. The following proposition provides an explicit formula for computing the \( K_3 \)-degree of a vertex in the complement of a graph.

\begin{proposition}[{\cite[Corollary 2.5]{Nair:94}}]~\label{K3-complement}
    If $G$ is an $r$-regular graph with $n$ vertices, then
    
    \[K_3 \deg_{\overline{G}}(u) = \binom{n - 1}{2} - \frac{3}{2}r(n - r - 1) - K_3 \deg_G(u).\]
\end{proposition}

The following result is an immediate consequence of Proposition~\ref{K3-complement} as the complement of any $r$-regular graph with $n$ vertices is $(n-r-1)$-regular. 

\begin{corollary}\label{reg-k3-irreg-complement}
    The complement of a regular $K_3$-irregular graph is itself a regular $K_3$-irregular graph.
\end{corollary}

The following result is an obvious generalization of the classical Handshaking Lemma to subgraph-based $F$-degrees.

\begin{lemma}
For any graph $G$, it holds
\begin{equation}\label{fdegs-sum}
\sum_{v \in V(G)} F \deg(v) = |V(F)| \cdot N(G, F),
\end{equation}
where $N(G,F)$ denotes the number of subgraphs in $G$ isomorphic to $F$.
\end{lemma}

\begin{corollary}~\label{triangle-shake-lemma}
The total sum of $K_3$-degrees over all vertices of any graph is divisible by $3$.
\end{corollary}

\section{Establishing bounds}\label{sect-bounds}

\subsection{The partitioning technique}\label{part-tech}

In this section, we introduce the \textit{partitioning technique} that will be abundantly used in further proofs.

Let $G$ be an $r$-regular $K_3$-irregular graph with $n$ vertices. Fix an arbitrary vertex $v \in V(G)$, and denote its triangle-degree by $d = K_3 \deg(v)$. Let $A=N(v)$ be its neighborhood, and let $B = V(G) \setminus N[v]$ denote all the remaining vertices. Thus, the vertex set $V(G)$ is partitioned into three parts: $V(G) = \{v\} \sqcup A \sqcup B$. 
Also, the edge set $E(G)$ is naturally divided into four parts:
\[E(G) = E(\{v\},A) \sqcup E(A) \sqcup E(A,B) \sqcup E(B).\]

Now, let us calculate the number of edges in each subset.
Consider the induced subgraph $G[A]$. Clearly, $|A| = r$ and $|E(A)| = d$. 
Since the graph $G$ is $r$-regular, by formula (\ref{fdegs-sum}), we have 
\[\sum_{v \in A} \deg(v) = |A| \cdot r = r^2.\] 
Next, we can calculate the number of edges between $A$ and $B$ by subtracting the edges incident to $v$ and those within $A$.

\begin{equation}
|E(A, B)| = r^2 - r - 2d = r(r-1) - 2d.
\end{equation}

From this, we can express the number of vertices and edges in $B$:
\begin{align}
    &|B| = n - r - 1, \label{vert-in-B}\\
    &|E(B)| =\frac{nr}{2} - r - d - (r(r-1) - 2d) = \frac{nr}{2} - r^2 + d \label{edg-in-B}.
\end{align}

See Figure~\ref{fig-2} for the visualization of the partitioning technique.

\begin{figure}[ht]
	\begin{tikzpicture}[auto,on grid, state/.style ={circle, top color =black , bottom color = black , draw, minimum width=1mm}, inner sep=1.5pt, scale = 1]

			\node[state, label={[xshift=-5pt]left:$v$}] (1) at (0, 0) {};
			
			\node[below left = 0.9cm and 0.35cm of 1] {\small $d = K_3 \deg(v)$};

			\draw[line width=0.5mm] (4, 0) ellipse (2cm and 2.5cm);
			\node[anchor=center] at (4, 0.5) {\small $|A| = r$};
			\node[anchor=center] at (4, -0.5) {\small $|E(A)| = d$};
			
			\node[below=1.25cm of 1] at (4, 4.5) {\large $A$};
			
			\draw[line width=0.5mm] (13, 0) ellipse (2cm and 2.5cm);
			\node[anchor=center] at (13, 0.5) {\small $|B| = n - r - 1$};
			\node[anchor=center] at (13, -0.5) {\small $|E(B)| = \frac{nr}{2} - r^2 + d$};
			
			\node[below=1.25cm of 1] at (13, 4.5) {\large $B$};

                \draw[line width=0.5mm] (6, -0.5) -- (11, -0.5);
			\node[] at (8.5, 0) {\large $\ldots$};
			\draw[line width=0.5mm] (6, 0.5) -- (11, 0.5);
			
			\node[] at (8.5, -1.5) {\small $|E(A,B)| = r(r - 1) - 2d$};
			
			\coordinate (T1) at (2.76, 1.95);
			\coordinate (T2) at (2.76, -1.95);
			\draw[line width=0.5mm] (1) -- (T1);
			\draw[line width=0.5mm] (1) -- (T2);
    \end{tikzpicture}
    \caption{The partitioning technique for $r$-regular $K_3$-irregular graphs.}\label{fig-2}
\end{figure}
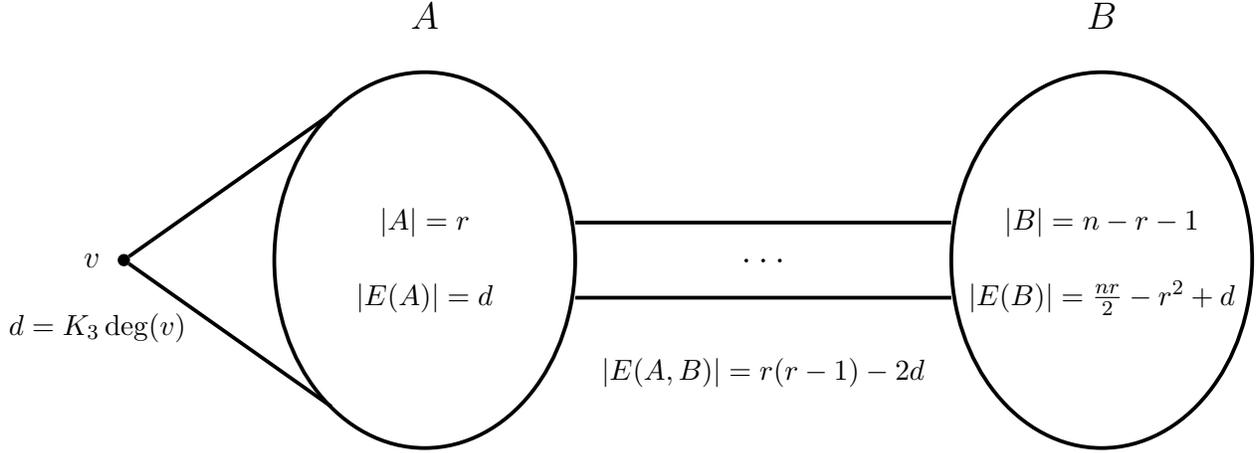

In the following results from this section, we always assume that $G$ is an $r$-regular $K_3$-irregular graph with subsets $A$ and $B$ defined as above. Now, we establish several structural properties in this setting. 

\begin{lemma}~\label{bounds-on-degs-in-A}
    For any vertex $a \in A$, it holds \[\deg_{G[A]}(a) \leq \min\{r-2, d\}.\]
\end{lemma}
\begin{proof}
    By the partitioning technique, the number of edges in $A$ is $|E(A)| = K_3 \deg(v) = d$. This implies that for every vertex $a \in A$, we have $\deg_{G[A]}(a) \leq d$. 
    Moreover, since $a$ and $v$ are adjacent by construction and the original graph is $r$-regular, we obtain $\deg_{G[A]}(a) \leq r-1$. If $\deg_{G[A]}(a) = r-1$, then $a$ and $v$ would be twins, implying $K_3 \deg(a) = K_3 \deg(v)$, which leads to a contradiction.
\end{proof}

\begin{corollary}~\label{no-iso-in-A-hat}
    The graph $\overline{G[A]}$ cannot have isolated vertices.
\end{corollary}

\begin{lemma}~\label{no-iso-in-B}
    The subgraph $G[B]$ cannot have isolated vertices.
\end{lemma}
\begin{proof}
    Suppose, for contradiction, that $w$ is an isolated vertex in $G[B]$. Since $w$ has no neighbors in $B$, it must be adjacent to all vertices in $A$, meaning $N(w) = A$. This implies that $w$ and $v$ (the anchor of the partition) are false twins, and thus $K_3 \deg (v) = K_3 \deg(w)$.
\end{proof}

\begin{lemma}~\label{bounds-on-degs-in-B}
    For any vertex $b \in B$, we have \[\max\{1, |B|-1-|E(\overline{G[B]})|\} \leq \deg_{G[B]}(b) \leq \min\{r, |E(B)|, |B|-1\}.\]
\end{lemma}
\begin{proof}
    The upper bound follows directly from the constraints given by $|E(B)|$, $r$, and $|B|$. For the lower bound, note that Lemma~\ref{no-iso-in-B} ensures that $b$ is not isolated, so we have ${\deg_{G[B]}(b) \geq 1}$. To establish the second lower bound, observe that it accounts for the case where $B$ contains many edges. The degree of $b \in B$ satisfies $\deg_{G[B]}(b) = |B| - 1 - \deg_{\overline{G[B]}}(b)$. Since $\deg_{\overline{G[B]}}(b) \leq |E(\overline{G[B]})|$, it follows that
    $\deg_{G[B]}(b) \geq |B| - 1 - |E(\overline{G[B]})|$.
\end{proof}

\subsection{Lower and upper bounds}

In this section, we provide bounds on the order $n$, the regularity $r$, and the $K_3$-degrees.

\begin{lemma}\label{K3_upper_bound}
Let $G$ be an $r$-regular $K_3$-irregular graph. Then, for any of its vertices $v\in V(G)$, the following upper bound holds:
\begin{equation}~\label{K3_upper_bound_formula}
    K_3 \deg(v) \leq \binom{r}{2} - \biggl\lceil{\frac{r}{2}\biggl\rceil}.
\end{equation}
\end{lemma}
\begin{proof}
To the contrary, assume that such a graph $G$ exists and has a vertex $v\in V(G)$ with $K_3 \deg(v) > \binom{r}{2} - \ceil{\frac{r}{2}}$. Consider the induced subgraph $H=G[N(v)]$, consisting of the neighborhood of $v$. Then, $|E(H)| > \binom{r}{2} - \ceil{\frac{r}{2}}$, and by the Pigeonhole Principle, $H$ must contain a universal vertex $w$. Hence, $v$ and $w$ are twins, implying $K_3 \deg(v) = K_3 \deg(w)$.
\end{proof}

Since all $K_3$-degrees are non-negative pairwise distinct integers, the following bound on the order of $G$ clearly follows.

\begin{corollary}\label{K3_upper_bound_n}
Let $G$ be an $r$-regular $K_3$-irregular graph with $n$ vertices. Then
\[n \leq \binom{r}{2} - \biggl\lceil{\frac{r}{2}\biggr\rceil} + 1.\]
\end{corollary}

Now, we establish a lower bound for the $K_3$-degrees of vertices in regular $K_3$-irregular graphs.

\begin{lemma}\label{K3_lower_bound}
Let $G$ be an $r$-regular $K_3$-irregular graph with $n$ vertices. Then, for any of its vertices $v\in V(G)$, the following lower bound holds:
\begin{equation}
    K_3 \deg(v) \geq \ceil{\frac{n-r-1}{2}} - \frac{nr}{2} + r^2.
\end{equation}
\end{lemma}
\begin{proof}
By Lemma~\ref{no-iso-in-B}, there are no isolated vertices in $G[B]$, which implies $|E(B)| \geq \ceil{\frac{|B|}{2}}$. By the partitioning technique, we have $|B| = n-r-1$, and $|E(B)| = \frac{nr}{2} - r^2 + K_3 \deg(v)$. By substituting and rearranging terms, we obtain the desired inequality.
\end{proof}

Let $M$ denote the maximum, and $m$ denote the minimum among all $K_3$-degrees of some regular $K_3$-irregular graph. Since all the $K_3$-degrees are distinct, we have $M - m + 1 \geq n$. Substituting the results of Lemma~\ref{K3_upper_bound} for~$M$ and Lemma~\ref{K3_lower_bound} for $m$, we obtain the next result.

\begin{corollary}\label{K3_both_bounds}
For any $r$-regular $K_3$-irregular graph with $n$ vertices, it holds
\begin{equation}
    \left(\binom{r}{2} + \frac{nr}{2} + 1\right) - \left(\biggl\lceil{\frac{r}{2}\biggr\rceil} + \ceil{\frac{n-r-1}{2}} + r^2\right) \geq n.
\end{equation}
\end{corollary}

\section[The non-existence of regular K\_3-irregular graphs for small regularities]{The non-existence of regular $\boldsymbol{K_3}$-irregular graphs \\ for small regularities}\label{sect-small-regs}

In this section, we investigate $r$-regular $K_3$-irregular graphs of small regularities. We begin with a proof that no such graphs exist for $r \leq 6$. In \cite{reg-triangle:24}, a brute-force search established that there are no regular $K_3$-irregular graphs up to $15$ vertices. This finding allows us to simplify some cases of our subsequent results.
Nevertheless, we provide a completely analytical proof that goes deeply into the structure of regular $K_3$-irregular graphs and demonstrates that they cannot exist for $r \leq 7$. Thus, we corroborate the computational findings and extend them even further. We begin by tackling the cases of regularities up to $6$ and then proceed to $r=7$.

\begin{proposition}\label{noless6reg}
No $r$-regular $K_3$-irregular graphs exist for $r \leq 6$.
\end{proposition}
\begin{proof}

The cases $r \in \{1, \ldots, 5\}$ are directly covered by Corollaries~\ref{K3_upper_bound_n} and~\ref{K3_both_bounds}. For example, for $r=5$, Corollary~\ref{K3_upper_bound_n} gives an upper bound $n \leq 8$, whereas Corollary~\ref{K3_both_bounds} gives a lower bound $n \geq 14$.

Now let $r = 6$. To obtain the lower bound on $n$, we again apply Corollary~\ref{K3_both_bounds}, and derive the inequality
\[\left( \binom{6}{2} + \frac{6n}{2} + 1\right) - \left(\ceil{\frac{6}{2}} + \ceil{\frac{n-6-1}{2}} + 6^2\right) \geq n.\]
Simplifying, we get
\[2n - \ceil{\frac{n-7}{2}} \geq 23.\]
From this, we deduce that $n \geq 13$. Next, applying Corollary~\ref{K3_upper_bound_n}, we find that $n \leq 13$. The combination of these constraints implies that $n=13$ with $K_3$-degrees ranging consecutively from $0$ to $12$.

Consider the partitioning technique (see Section~\ref{part-tech}) with respect to the vertex $v$ having $K_3 \deg(v) = 12$. Then, by (\ref{vert-in-B}) and (\ref{edg-in-B}), we get
\begin{align*}
|B| &= n - r - 1 = 13 - 6 - 1 = 6\text{, and}\\
|E(B)| &=\frac{nr}{2} - r^2 + d = 39 - 36 + 12 = 15\text{, implying that $G[B] \cong K_6$.}
\end{align*}
Now, it follows that each vertex $b \in B$ has $\deg_{G[B]}(b)=5$, and therefore must be adjacent to exactly one vertex in $A$. 

Suppose that there exists a pair of distinct vertices $v,u \in B$ sharing a common neighbor from $A$, that is
\[{|N(v) \cap N(u) \cap A| = 1}.\] 
Then $v$ and $u$ are twins and have ${K_3 \deg(v) = K_3 \deg(u)}$. As a result, for every pair of distinct vertices $v,u \in B$, it holds
\[N(v) \cap N(u) \cap A = \varnothing.\] 
However, in this case, each vertex $b \in B$ has
\[K_3 \deg_G(b) = K_3 \deg_{G[B]}(b) = \binom{5}{2} = 10.\] 
This contradicts $G$ being $K_3$-irregular, and finishes the proof.
\end{proof}

Now we present our main result for this section.

\begin{theorem}\label{no7reg}
    No $7$-regular $K_3$-irregular graph exists.
\end{theorem}
\begin{proof}
Substituting $r=7$ into Corollary~\ref{K3_both_bounds}, we obtain the lower bound on the order $n \geq 14$. Meanwhile, Lemma~\ref{K3_upper_bound} implies that $K_3 \deg(v) \leq 17$ for every vertex $v \in V(G)$. Consequently, $n \leq 18$. Since the regularity is odd, we only need to consider even values for $n$, that is $n \in \{14, 16, 18\}$. We examine each case separately.

\paragraph{\large \emph{Case $n=14$}:}
By Lemma~\ref{K3_lower_bound}, for every vertex $v \in V(G)$,
\[K_3 \deg(v) \geq \ceil{\frac{14-7-1}{2}} - \frac{7 \cdot 14}{2} + 7^2 = 3.\] 
Thus, the $K_3$-degrees lie in the range $3 \leq K_3 \deg(v) \leq 17$. Consider the partitioning technique (see Section~\ref{part-tech}) with respect to a vertex $v$ such that $K_3 \deg(v) = 17$. This leads to $|B| = 6$ and ${|E(B)| = 17}$. However, the maximum number of edges in a graph on $6$ vertices is ${|E(K_6)| = 15}$, which is a contradiction. Thus every vertex $v \in V(G)$ satisfies $3 \leq K_3 \deg(v) \leq 16$, which yields exactly $14$ possible $K_3$-degrees, but this case is ruled out by Corollary~\ref{triangle-shake-lemma} as $\sum_{k=3}^{16}k = 133$.

\paragraph{\large \emph{Case $n=16$}:} 
We will first show that a vertex $v$ with $K_3 \deg(v) = 0$ cannot exist. Consider the partitioning technique with respect to $v$. We obtain $|B| = 8$ and $|E(B)| = 7$.

Since $|E(A)|=0$, the $K_3$-degree of every vertex $a \in A$ is bounded above by $K_3 \deg(a) \leq |E(B)|$. According to Lemma~\ref{no-iso-in-B}, $G[B]$ cannot have isolated vertices. Additionally, every vertex from $A$ has 6 neighbors from $B$ and $|B|$ = 8. This ensures that at least one edge from $E(B)$ is not fully contained in $G[N[a]]$, leading to a stricter bound: ${K_3 \deg(a) \leq |E(B)| - 1 = 6}$. Therefore, for all $a \in A$ we have $1 \leq K_3 \deg(a) \leq 6$. However, $A$ has $7$ vertices and, by the Pigeonhole Principle, it is a contradiction.
    
Now, let us show that a vertex $v$ with $K_3 \deg(v) = 17$ cannot exist in this case either. Consider the partitioning technique for this vertex.
We have $|B| = 8$ and $|E(B)|=24$. Since the complete graph with $8$ vertices has $|E(K_8)| = 28$ edges, the structure of $G[B]$ is $K_8$ with $4$ edges removed. Moreover, $G[B]$ cannot contain a universal vertex $u$, as that would imply $K_3 \deg(u) = 17 = K_3 \deg(v)$. Therefore, each vertex in $B$ must be nonadjacent to at least one other vertex in $B$, meaning $\overline{G[B]} \cong 4K_2$.
Since every vertex of $B$ plays the same role, we obtain 
\[K_3 \deg(b) \geq \binom{|B|-2}{2} - 3 = 12,\] 
for every $b \in B$.
Thus, the vertices in $B$ have $K_3$-degrees between $12$ and $16$. Since there are $8$ vertices but only $5$ distinct possible $K_3$-degrees, by the Pigeonhole Principle, we get a contradiction.

Now, we have shown that if $n=16$, then $1 \leq K_3 \deg(v) \leq 16$ for all $v \in V(G)$. 
However, the total sum of $K_3$-degrees $\sum_{k=1}^{16} k = 136$ is not divisible by $3$, which contradicts the necessary divisibility condition of Corollary~\ref{triangle-shake-lemma}, and rules out this case.

\paragraph{\large \emph{Case $n=18$}:}
We will first show that a vertex $v$ with $K_3 \deg(v) = 17$ cannot exist. Applying the partitioning technique to the vertex $v$ yields $|A| = 7$, and $|E(A)| = 17$; consequently, $|E(\overline{G[A]})| = 4$. By Corollary~\ref{no-iso-in-A-hat}, the subgraph $\overline{G[A]}$ has no isolated vertices, therefore $\overline{G[A]} \cong 2K_2 \cup P_3$. 
For every vertex $a \in A$ with $\deg_{\overline{G[A]}}(a) = 1$, we have 
\[K_3 \deg(a) \geq K_3 \deg_{G[A \cup \{v\}]}(a) \geq \binom{6}{2} - 3 = 12.\] 
Therefore, the $K_3$-degrees of these six vertices are at least $12$ and at most $16$, which is impossible.
\end{proof}

\section{The curious case of 8-regular graphs}\label{sect-8-reg}

In this section, we establish bounds on the order of $8$-regular $K_3$-irregular graphs, narrowing down the possibilities to six cases: $17 \leq n \leq 22$. We start with settling the lower bound.

\begin{proposition}
    Any $8$-regular $K_3$-irregular graph has order $n \geq 17$.
\end{proposition}
\begin{proof}
     When $n \leq 16$, the complement is an $r'$-regular $K_3$-irregular graph with $r' \leq 7$ (see Corollary~\ref{reg-k3-irreg-complement}) which contradicts Proposition~\ref{noless6reg} and Theorem~\ref{no7reg}.
\end{proof}

Next, we proceed with deriving an upper bound for the $K_3$-degrees of $8$-regular $K_3$-irregular graphs.

\begin{theorem}\label{8-reg-upper-bound-k3}
    If an $8$-regular $K_3$-irregular graph $G$ exists, then for every $v \in V(G)$, it holds \[K_3 \deg(v) \leq 22.\]
\end{theorem}
\begin{proof}

At first, we use Lemma~\ref{K3_upper_bound} to obtain ${K_3 \deg(v) \leq 24}$. Next, we lower this bound down to $22$ in two steps.

\paragraph{\large \textit{Step 1}: $\boldsymbol{K_3 \deg(v) = 24}$ is impossible.}
Towards a contradiction, assume that there is a vertex $v$ having $K_3 \deg(v) = 24$. Consider the partitioning technique (see Section~\ref{part-tech}) with respect to $v$ in order to get $|A|=8$ and $|E(\overline{G[A]})|=4$. 
By Corollary~\ref{no-iso-in-A-hat}, there are no isolated vertices in $\overline{G[A]}$, therefore $\overline{G[A]} \cong 4K_2$.
Thus, every vertex $a \in A$ has 
\[K_3 \deg(a) \geq \binom{7}{2} - 3 = 18.\] 
In summary, there are only $6$ available values in the range $18 \leq K_3 \deg(a) \leq 23$, yet $|A| = 8$. By the Pigeonhole Principle, this leads to a contradiction, and so $K_3 \deg(v) \leq 23$, for all $v \in V(G)$.

\paragraph{\large \textit{Step 2}: $\boldsymbol{K_3 \deg(v) = 23}$ is impossible.}
Again, towards a contradiction, assume that there exists a vertex $v$ having $K_3 \deg(v) = 23$. Consider the partitioning technique with respect to $v$ in order to get $|A|=8$ and $|E(\overline{G[A]})|=5$. Also, note that by Corollary~\ref{no-iso-in-A-hat}, the graph $\overline{G[A]}$ does not contain isolated vertices. 

Now define the subset \[W= \{ w \in A\mid \deg_{\overline{G[A]}}(w)=1\},\] which consists of the vertices from $A$ that have a unique neighbor in $\overline{G[A]}$. In other words, $\deg_{G[A]}(w)=6$ for all $w \in W$. 

\textbf{Claim 1:} $|W| \geq 6$. \\
Indeed, assuming $|W|\leq 5$, we obtain 
\begin{align*}
|E(\overline{G[A]})|&\geq\frac{1}{2}(|W|+2|A\setminus W|) = \frac{1}{2}(|W| + 2(|A|-|W|))\\
&= \frac{1}{2}(2|A|-|W|) = \frac{1}{2}(16-|W|)\geq\frac{11}{2}>5,
\end{align*}
which is a contradiction. This concludes Claim 1.

Fix an arbitrary $w \in W$ and denote by $w'$ its unique neighbor in $\overline{G[A]}$. 
In the next claim, we establish a formula for the $K_3$-degree of $w$ in $G[A \cup \{v\}]$ in terms of the degree of $w'$ in $\overline{G[A]}$.

\textbf{Claim 2:}
\begin{equation}\label{eqK3degA23}
	K_3 \deg_{G[A\cup\{v\}]}(w) = 16 + \deg_{\overline{G[A]}}(w').
\end{equation}

Since $w \in W$, the vertex $w$ is adjacent to all of $A \setminus \{w, w'\}$ and to $v$ in $G$, giving $\deg_{G[A \cup \{v\}]}(w) = |A| - 2 + 1 = 7$ neighbors.

Every triangle containing $w$ in $G[A \cup \{v\}]$ is formed by a pair of its neighbors that are themselves adjacent. Thus, $K_3 \deg_{G[A \cup \{v\}]}(w)$ equals $\binom{7}{2}$ minus the number of non-edges among the $7$ neighbors of $w$. Since $v$ is adjacent to every vertex in $A$, the only non-edges among $w$'s neighbors lie entirely within $A \setminus \{w,w'\}$, i.e., they are exactly the edges of $\overline{G[A]}$ not incident to $w$ or $w'$. The number of such edges is 
\[|E(\overline{G[A]})| - \deg_{\overline{G[A]}}(w) - (\deg_{\overline{G[A]}}(w') - 1) = 5 - \deg_{\overline{G[A]}}(w').\] 
Therefore, we have proved the Claim 2,
\[ K_3 \deg_{G[A\cup\{v\}]}(w) = \binom{7}{2} - (5 - \deg_{\overline{G[A]}}(w')) = 16 + \deg_{\overline{G[A]}}(w').\]

\textbf{Claim 3: $\overline{G[A]} \cong 2K_2 \cup P_4$ or $\overline{G[A]} \cong K_2 \cup 2P_3$.}

Since $\deg_{\overline{G[A]}}(w') \geq 1$ (there are no isolated vertices in $\overline{G[A]}$), we obtain $K_3 \deg_{G}(w) \geq K_3 \deg_{G[A\cup\{v\}]}(w) \geq 17$ for every $w \in W$.
Since there are at least $6$ such vertices, it follows that all vertices with $K_3$-degrees ranging from $17$ to $22$ lie in $W \subseteq A$. From this we also conclude that $|W|=6$.

Hence the graph $\overline{G[A]}$ has $8$ vertices, $5$ edges, no isolated vertices, and exactly $|W|=6$ leaf vertices. It can be easily seen (for example, by considering the degree sequence) that $\overline{G[A]}$ is isomorphic to either $2K_2 \cup P_4$ or $K_2 \cup 2P_3$. This concludes Claim 3.

Recall that if two vertices in a graph are true (or false) twins, then they have equal $K_3$-degrees. Note that every vertex $u \in W$ has $\deg_{G[A \cup \{v\}]}(u)=7$ and $u$ has exactly one neighbor from $B$. Consequently, if two vertices $u_1, u_2 \in W$ are true (or false) twins in $G[A\cup\{v\}]$, and they share a common neighbor $b \in B$, then $u_1$ and $u_2$ are true (or false) twins in $G$.

Now, consider the vertex $w \in W$ such that $K_3 \deg(w) = 22$. The neighborhood of $w$ consists of the following elements: 
\begin{itemize}
	\item $6$ vertices of $A$ (by the definition of $W$); 
	\item the vertex $v$ (the ``anchor'' of the partition); 
	\item $1$ vertex from $B$, which we denote by $b$.
\end{itemize}

Further, we consider four possible cases for the vertex $w$ depending on the structure of the graph $\overline{G[A]}$ and the degree of $w'$.

\textbf{Case 1 (a):} $\overline{G[A]} \cong K_2 \cup 2P_3$ and $\deg_{\overline{G[A]}}(w') = 1$.

Here the edge $ww'$ is a $K_2$-component of $\overline{G[A]}$. By~(\ref{eqK3degA23}), 
\[K_3 \deg_{G[A\cup\{v\}]}(w) = 16 + 1 = 17.\]

\begin{figure}
	\centering
	\begin{subfigure}[b]{0.4\textwidth}
		\centering
		\begin{tikzpicture}[
			lf/.style={circle, draw, fill=black, minimum size=5pt, inner sep=0pt},
			nl/.style={circle, draw, fill=white, minimum size=5pt, inner sep=0pt}]
			\node[lf, darkgreen, label=above:{\small $p_1$}] (w) at (0, 1) {};
			\node[nl, label=above:{\small $p_2$}] (wp) at (1, 1) {};
			\node[lf, darkgreen, label=above:{\small $p_3$}] (a1) at (2, 1) {};
			\draw[dashed] (w) -- (wp);
			\draw[dashed] (wp) -- (a1);
			\node[lf, blue, label=above:{\small $q_1$}] (b1) at (0, 0) {};
			\node[nl, label=above:{\small $q_2$}] (b2) at (1, 0) {};
			\node[lf, blue, label=above:{\small $q_3$}] (b3) at (2, 0) {};
			\draw[dashed] (b1) -- (b2) -- (b3);
			\node[lf, red, label=above:{\small $w$}] (c1) at (0.5, -1) {};
			\node[lf, red, label=above:{\small $w'$}] (c2) at (1.5, -1) {};
			\draw[dashed, line width=0.6mm] (c1) -- (c2);
		\end{tikzpicture}
		\caption{$w$ is a leaf of $K_2$}\label{fig-k2-2p3-a}
	\end{subfigure}
	\hfill
	\begin{subfigure}[b]{0.4\textwidth}
		\centering
		\begin{tikzpicture}[
			lf/.style={circle, draw, fill=black, minimum size=5pt, inner sep=0pt},
			nl/.style={circle, draw, fill=white, minimum size=5pt, inner sep=0pt}]
			\node[lf, darkgreen, label=above:{\small $w$}] (w) at (0, 1) {};
			\node[nl, label=above:{\small $w'$}] (wp) at (1, 1) {};
			\node[lf, darkgreen, label=above:{\small $p_3$}] (a1) at (2, 1) {};
			\draw[dashed, line width=0.6mm] (w) -- (wp);
			\draw[dashed] (wp) -- (a1);
			\node[lf, blue, label=above:{\small $q_1$}] (b1) at (0, 0) {};
			\node[nl, label=above:{\small $q_2$}] (b2) at (1, 0) {};
			\node[lf, blue, label=above:{\small $q_3$}] (b3) at (2, 0) {};
			\draw[dashed] (b1) -- (b2) -- (b3);
			\node[lf, red, label=above:{\small $x$}] (c1) at (0.5, -1) {};
			\node[lf, red, label=above:{\small $y$}] (c2) at (1.5, -1) {};
			\draw[dashed] (c1) -- (c2);
		\end{tikzpicture}
		\caption{$w$ is a leaf of $P_3$}\label{fig-k2-2p3-b}
	\end{subfigure}
	\caption{Two possible positions of the edge $ww'$ in $\overline{G[A]} \cong K_2 \cup 2P_3$. Filled vertices belong to~$W$; open vertices do not. Dashed edges are the edges of~$\overline{G[A]}$.}
\end{figure}

Further, we split the $22$ triangles that contain $w$ into two classes:
\begin{itemize}
    \item Triangles involving $v$ and vertices of $A$ --- there are exactly $17$ of them.  
    \item Triangles containing the vertex $b \in B$ --- the remaining $5$ triangles.
\end{itemize}

Therefore, $b$ must be adjacent to $w$ and to $5$ of its $6$ neighbors in $A$. 
Let $p_1 \text{-} p_2 \text{-} p_3$, $q_1 \text{-} q_2 \text{-} q_3$, $w \text{-} w'$ be the connected components of $\overline{G[A]}$, see Figure~\ref{fig-k2-2p3-a}. The vertex $b$ must be adjacent to exactly $5$ vertices from the set $N_{G[A]}(w) = \{p_1, p_2, p_3, q_1, q_2, q_3\}$. However, both edges $p_1b$ and $p_3b$ cannot exist simultaneously: the vertices $p_1$ and $p_3$ satisfy $N_{G[A \cup \{v\}]}[p_1] = N_{G[A \cup \{v\}]}[p_3]$, so if both were adjacent to $b$ they would be true twins in $G$ (contradicting $K_3$-irregularity). Similarly, $q_1b$ and $q_3b$ cannot coexist. Therefore, $b$ can be adjacent to at most $6-2=4 < 5$ vertices of $N_{G[A]}(w)$, which is a contradiction.

\textbf{Case 1 (b):} $\overline{G[A]} \cong K_2 \cup 2P_3$ and $\deg_{\overline{G[A]}}(w') = 2$.

Here the edge $ww'$ is an edge in a $P_3$-component of $\overline{G[A]}$. By~(\ref{eqK3degA23}), 
\[K_3 \deg_{G[A\cup\{v\}]}(w) = 16 + 2 = 18.\]

Further, we split the $22$ triangles that contain $w$ into two classes:
\begin{itemize}
	\item Triangles involving $v$ and vertices of $A$ --- there are exactly $18$ of them.  
	\item Triangles containing the vertex $b \in B$ --- the remaining $4$ triangles.
\end{itemize}

Thus, $b$ must be adjacent to $w$ and to $4$ of its $6$ neighbors in $A$. Let $w \text{-} w' \text{-} p_3$, $q_1 \text{-} q_2 \text{-} q_3$, $x \text{-} y$ be the connected components of $\overline{G[A]}$, see Figure~\ref{fig-k2-2p3-b}. First, note that $p_3b \notin E(G)$, as otherwise $w$ and $p_3$ would be true twins in $G$. Thus, $b$ must choose exactly  $4$ vertices from the set $M=\{q_1, q_2, q_3, x, y\}$. However, both edges $q_1b$ and $q_3b$ cannot exist simultaneously (otherwise $q_1$ and $q_3$ become true twins in $G$); similarly, both edges $xb$ and $yb$ cannot exists simultaneously (otherwise $x$ and $y$ become false twins in $G$). Therefore, $b$ can be adjacent to at most $3$ vertices from the set $M$, which is a contradiction. 

\textbf{Case 2 (a):} $\overline{G[A]} \cong 2K_2 \cup P_4$ and $\deg_{\overline{G[A]}}(w') = 1$.

Here the edge $ww'$ is a $K_2$-component of $\overline{G[A]}$. 
By~(\ref{eqK3degA23}), 
\[K_3 \deg_{G[A\cup\{v\}]}(w) = 16 + 1 = 17.\] 

Further, we split the $22$ triangles that contain $w$ into two classes:
\begin{itemize}
	\item Triangles involving $v$ and vertices of $A$ --- there are exactly $17$ of them.  
	\item Triangles containing the vertex $b \in B$ --- the remaining $5$ triangles.
\end{itemize}

\begin{figure}[ht]
	\centering
	\begin{subfigure}[b]{0.4\textwidth}
		\centering
		\begin{tikzpicture}[
			lf/.style={circle, draw, fill=black, minimum size=5pt, inner sep=0pt},
			nl/.style={circle, draw, fill=white, minimum size=5pt, inner sep=0pt}]
			\node[lf, red, label=above:{\small $w$}] (w) at (2.2, 1) {};
			\node[lf, red, label=above:{\small $w'$}] (wp) at (3.2, 1) {};
			\draw[dashed, line width=0.6mm] (w) -- (wp);
			\node[lf, blue, label=above:{\small $z$}] (a1) at (0, 1) {};
			\node[lf, blue, label=above:{\small $t$}] (t) at (1, 1) {};
			\draw[dashed] (a1) -- (t);
			\node[lf, label=below:{\small $p_1$}] (b1) at (0, -1) {};
			\node[nl, label=below:{\small $p_2$}] (b2) at (1, -1) {};
			\node[nl, label=below:{\small $p_3$}] (b3) at (2.2, -1) {};
			\node[lf, label=below:{\small $p_4$}] (b4) at (3.2, -1) {};
			\draw[dashed] (b1) -- (b2) -- (b3) -- (b4);
			
			\node[lf, label=above:{\small $b$}] (b) at (5.5, 0) {};
			
			\draw (b) -- (w);
			\draw (b) -- (a1);
			\draw (b) -- (b1);
			\draw (b) -- (b2);
			\draw (b) -- (b3);
			\draw (b) -- (b4);
		\end{tikzpicture}
		\caption{$w$ is a leaf of $K_2$}\label{fig-2k2-p4-a}
	\end{subfigure}
	\hfill
	\begin{subfigure}[b]{0.4\textwidth}
		\centering
		\begin{tikzpicture}[
			lf/.style={circle, draw, fill=black, minimum size=5pt, inner sep=0pt},
			nl/.style={circle, draw, fill=white, minimum size=5pt, inner sep=0pt}]
			\node[lf, red, label=above:{\small $x$}] (b1) at (0, 1) {};
			\node[lf, red, label=above:{\small $y$}] (b2) at (1, 1) {};
			\draw[dashed] (b1) -- (b2);
			\node[lf, blue, label=above:{\small $z$}] (c1) at (2.2, 1) {};
			\node[lf, blue, label=above:{\small $t$}] (c2) at (3.2, 1) {};
			\draw[dashed] (c1) -- (c2);
			
			\node[lf, label=below:{\small $w$}]  (w)  at (0, -1) {};
			\node[nl, label=below:{\small $w'$}] (wp) at (1, -1) {};
			\node[nl, label=below:{\small $p_3$}] (a1) at (2.2, -1) {};
			\node[lf, label=below:{\small $p_4$}] (a2) at (3.2, -1) {};
			\draw[dashed, line width=0.6mm] (w) -- (wp);
			\draw[dashed] (wp) -- (a1) -- (a2);
			
			\node[lf, label=above:{\small $b$}] (b) at (5.5, 0) {};
			
			\draw (b) -- (w);
			\draw (b) -- (b2);
			\draw (b) -- (c1);
			\draw (b) -- (a1);
			\draw (b) -- (a2);
		\end{tikzpicture}
		\caption{$w$ is a leaf of $P_4$}\label{fig-2k2-p4-b}
	\end{subfigure}
	\caption{Two possible positions of the edge $ww'$ in $\overline{G[A]} \cong 2K_2 \cup P_4$. Filled vertices belong to~$W$; open vertices do not. Dashed edges represent edges of~$\overline{G[A]}$; solid edges connect vertices to~$b$ in~$G$.}
\end{figure}

Thus, $b$ must be adjacent to $w$ and to $5$ of its $6$ neighbors in $A$. Let $z\text{-}t$, $w\text{-}w'$ be the two $K_2$ components and $p_1 \text{-} p_2 \text{-} p_3 \text{-} p_4$ the $P_4$-component of $\overline{G[A]}$, see Figure~\ref{fig-2k2-p4-a}. Note that at most one edge from $\{zb, tb\}$ can exist (otherwise $z$ and $t$, being false twins in $G[A \cup \{v\}]$, would become false twins in $G$). Without loss of generality, $N(b) \cap A = \{p_1, p_2, p_3, p_4, z, w\}$ (note that $w'b \notin E(G)$, since otherwise, $w$ and $w'$ would be false twins in $G$).

Since $z \in W$ and its unique $\overline{G[A]}$-neighbor is $t$ with $\deg_{\overline{G[A]}}(t)=1$, by~(\ref{eqK3degA23}),
\[K_3 \deg_{G[A \cup \{v\}]}(z) = 16 + 1 = 17.\]
Moreover, $z$ has exactly $5$ common neighbors with $b$ in $A$, each giving a new triangle. Therefore, $K_3 \deg_G(z) = 17 + 5 = 22$, which contradicts $K_3$-irregularity.

\textbf{Case 2 (b):} $\overline{G[A]} \cong 2K_2 \cup P_4$ and $\deg_{\overline{G[A]}}(w') = 2$.

Here $w$ is a leaf vertex of the $P_4$ component of $\overline{G[A]}$, and $w'$ is its neighbor. 
By~(\ref{eqK3degA23}), 
\[K_3 \deg_{G[A\cup\{v\}]}(w) = 16 + 2 = 18.\] 

Further, we split the $22$ triangles that contain $w$ into two classes:
\begin{itemize}
	\item Triangles involving $v$ and vertices of $A$ --- there are exactly $18$ of them.  
	\item Triangles containing the vertex $b \in B$ --- the remaining $4$ triangles.
\end{itemize}

Thus, $b$ must be adjacent to $w$ and to $4$ of its $6$ neighbors in $A$. Let $x \text{-} y$, $z \text{-} t$ be the two $K_2$-components and $w \text{-} w' \text{-} p_3 \text{-} p_4$ the $P_4$-component of $\overline{G[A]}$, see Figure~\ref{fig-2k2-p4-b}. At most one edge from $\{xb, yb\}$ can exist (otherwise $x$ and $y$ become false twins in $G$); similarly, at most one edge from $\{zb, tb\}$ can exist. Without loss of generality, $\{y, z, p_3, p_4\} \subset N(b)$ (note, that the edge $w'b$ may or may not exist, but it does not contribute a triangle with $w$).

Apply (\ref{eqK3degA23}) to $y,z \in W$ in order to obtain
\[K_3 \deg_{G[A \cup \{v\}]}(y) = K_3 \deg_{G[A \cup \{v\}]}(z) = 17.\]

Moreover, $y$ has at least $4$ common neighbors with $b$ in $A$, 
namely $\{w, z, p_3, p_4\}$, and,
likewise, we have $\{w, y, p_3, p_4\}$ for $z$. 
Each of these common neighbors yields a new triangle.
Therefore, 
\[K_3 \deg_{G}(y) \geq 17 + 4 = 21 \quad \text{and} \quad K_3 \deg_{G}(z) \geq 17 + 4 = 21.\]
Since $K_3 \deg(v)=23$ and $K_3 \deg(w)=22$, the only value available for both $y$ and $z$ is $21$, which contradicts $K_3$-irregularity.

Therefore, each vertex $v\in V(G)$ must have $K_{3} \deg(v)\leq 22$.
\end{proof}

\begin{corollary}
    If an $8$-regular $K_3$-irregular graph of order $n$ exists, then $n \leq 22$.
\end{corollary}
\begin{proof}
By Theorem~\ref{8-reg-upper-bound-k3}, we have $0 \leq K_3 \deg(v) \leq 22$ for every $v \in V(G)$, which directly implies $n \leq 23$. Now, assume $n=23$; in this case, all integers from $0$ to $22$ must appear as distinct $K_3$-degrees. However, the total sum $\sum_{k=0}^{22} k = 253$ is not divisible by $3$. Hence, it contradicts the necessary divisibility condition from Corollary~\ref{triangle-shake-lemma}. Therefore, such a graph cannot exist.
\end{proof} 

This leads to our final bounds of this section.

\begin{corollary}
    If an $8$-regular $K_3$-irregular graph of order $n$ exists, then $17 \leq n \leq 22$, and every vertex $v \in V(G)$ satisfies $0 \leq K_3 \deg(v) \leq 22$.
\end{corollary}

Using a heuristic search, the best we found for the regularity $r=8$ are the $n$-vertex graphs with exactly $2$ equal pairs of $K_3$-degrees for $n \in \{19,20,21,22\}$. We give an example of one of these graphs in Figure~\ref{close-8-reg-k3-irreg-figure}.

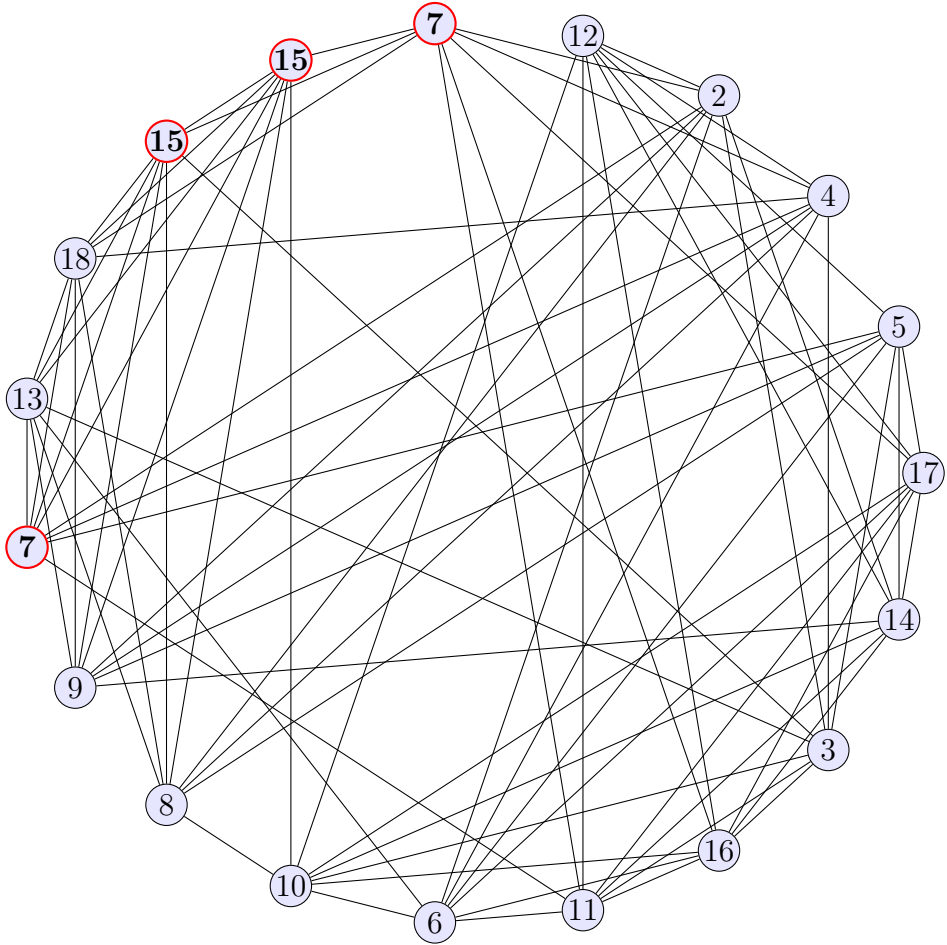
\begin{figure}[ht]
\centering
\begin{tikzpicture}[scale=1, every node/.style={circle, draw, fill=blue!10, minimum size=5.5mm, inner sep=0pt}]
    \def\labels{{17,5,4,2,12,7,15,15,18,13,7,9,8,10,6,11,16,3,14}}
    
	\foreach \i in {0,...,18} {
        \pgfmathparse{\labels[\i]}
        \let\labelval\pgfmathresult

        \ifthenelse{\labelval = 7 \OR \labelval = 15}
        {
		\node[draw=red, thick] (\i) at ({360/19 * \i}:6) {\textbf{\labelval}};
        }
        {
        \node (\i) at ({360/19 * \i}:6) {\labelval};
        }
	}
	
	\foreach \i/\neighbors in {
		0/{1,4,5,13,14,15,16,18},
        1/{0,4,10,11,12,14,17,18},
        2/{4,5,8,10,11,12,14,17},
        3/{4,5,10,11,12,14,17,18},
        4/{0,1,2,3,13,15,16,18},
        5/{0,2,3,6,7,8,15,16},
        6/{5,7,8,9,10,11,12,13},
        7/{5,6,8,9,10,11,12,17},
        8/{2,5,6,7,9,10,11,12},
        9/{6,7,8,10,11,12,14,17},
        10/{1,2,3,6,7,8,9,15},
        11/{1,2,3,6,7,8,9,18},
        12/{1,2,3,6,7,8,9,13},
        13/{0,4,6,12,14,16,17,18},
        14/{0,1,2,3,9,13,15,16},
        15/{0,4,5,10,14,16,17,18},
        16/{0,4,5,13,14,15,17,18},
        17/{1,2,3,7,9,13,15,16},
        18/{0,1,3,4,11,13,15,16}}
	{
		\foreach \j [evaluate=\j as \jnum using int(\j)] in \neighbors {
			\ifnum\i<\j
			\draw (\i) -- (\j);
			\fi
		}
	}
\end{tikzpicture}
\caption{An $8$-regular graph of order $19$. Two pairs of its vertices have equal $K_3$-degrees ($7$ and $15$). Vertex labels correspond to their $K_3$-degrees.}\label{close-8-reg-k3-irreg-figure}
\end{figure}

We believe that these bounds can be improved with some deeper analysis of the structure of such graphs, and hence, make the following conjecture.

\begin{conjecture}
    There are no $8$-regular $K_3$-irregular graphs.
\end{conjecture}

\section{Regularity 9 and beyond}\label{sect-9-reg-plus}

In this section, we discuss the experimental findings of our work. While the next section provides technical details of the algorithm, here we focus solely on the discovered graphs.

The smallest regularity for which we successfully found a regular $K_3$-irregular graph is $r=9$. One such graph is shown in Figure~\ref{9-reg-k3-irreg-figure}, with its adjacency lists representation given in Table~\ref{tab:9reg-k3-degrees-neighbors}.

\begin{figure}
	\centering
\begin{tikzpicture}[scale=1, every node/.style={circle, draw, fill=blue!10, minimum size=5.5mm, inner sep=0pt}]
	\foreach \i in {3,...,26} {
		\node (\i) at ({-45 + 360/24 * \i}:7) {\i};
	}
	
	\foreach \i/\neighbors in {
		3/{4,7,8,10,11,12,19,21,22},
		4/{3,5,6,9,13,16,17,18,20},
		5/{4,7,8,10,11,12,15,23,25},
		6/{4,7,8,10,12,14,21,23,26},
		7/{3,5,6,9,13,14,16,17,18},
		8/{3,5,6,9,13,14,15,16,17},
		9/{4,7,8,11,19,21,24,25,26},
		10/{3,5,6,11,15,16,17,18,20},
		11/{3,5,9,10,13,14,16,17,18},
		12/{3,5,6,13,14,15,16,17,18},
		13/{4,7,8,11,12,15,16,18,19},
		14/{6,7,8,11,12,15,17,18,24},
		15/{5,8,10,12,13,14,16,18,22},
		16/{4,7,8,10,11,12,13,15,17},
		17/{4,7,8,10,11,12,14,16,18},
		18/{4,7,10,11,12,13,14,15,17},
		19/{3,9,13,20,22,23,24,25,26},
		20/{4,10,19,21,22,23,24,25,26},
		21/{3,6,9,20,22,23,24,25,26},
		22/{3,15,19,20,21,23,24,25,26},
		23/{5,6,19,20,21,22,24,25,26},
		24/{9,14,19,20,21,22,23,25,26},
		25/{5,9,19,20,21,22,23,24,26},
		26/{6,9,19,20,21,22,23,24,25}}
	{
		\foreach \j [evaluate=\j as \jnum using int(\j)] in \neighbors {
			\ifnum\i<\j
			\draw (\i) -- (\j);
			\fi
		}
	}
\end{tikzpicture}
	\caption{A discovered $9$-regular $K_3$-irregular graph of order $24$. Vertex labels correspond to their $K_3$-degrees.}
    \label{9-reg-k3-irreg-figure}
\end{figure}
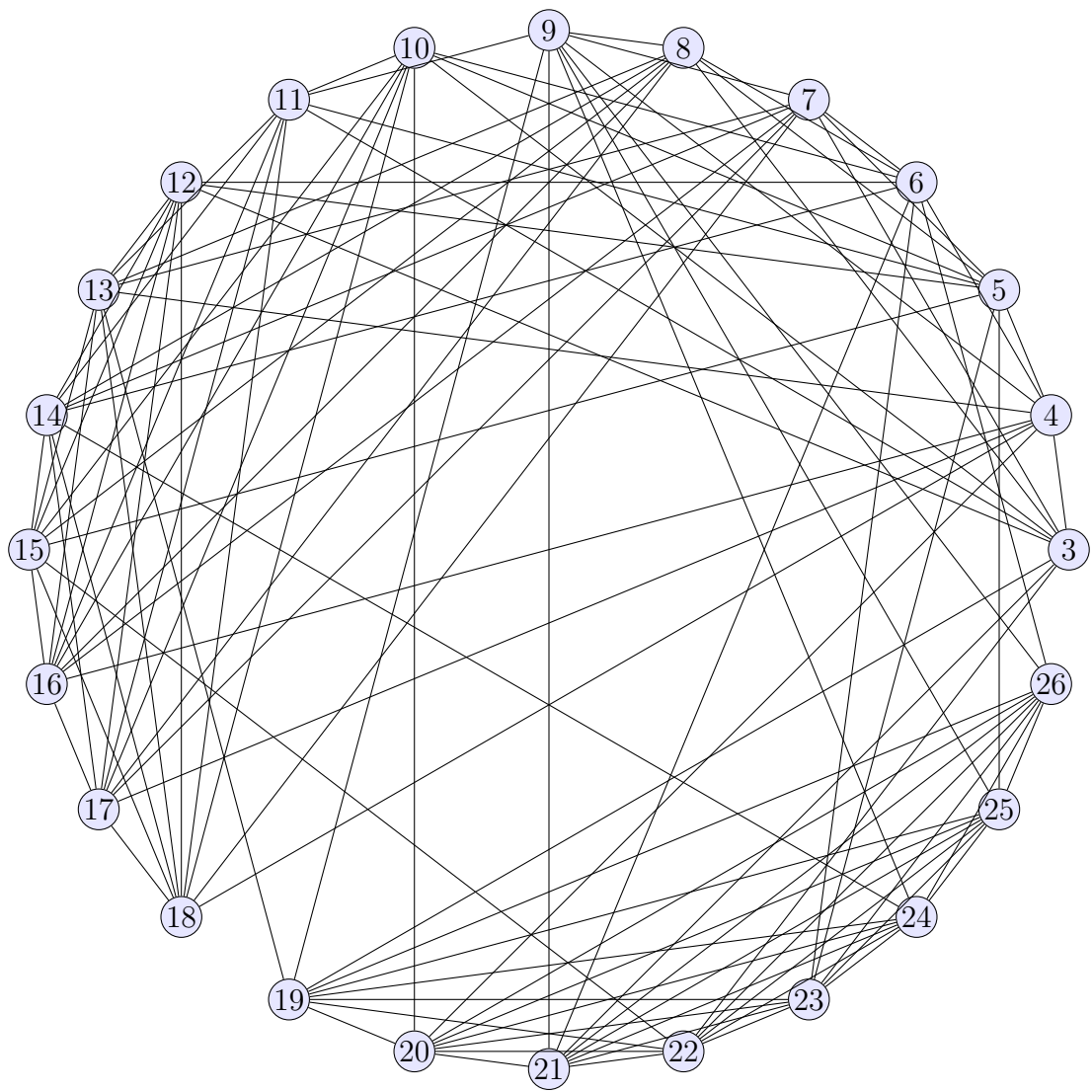

\begin{table}[!ht]
	\centering
    \renewcommand{\arraystretch}{1.0}
	\begin{tabular}[t]{| c | l |}
		\hline
		Vertex $v$ & Neighbors $N(v)$ \\
        \hline
        3  & \{4, 7, 8, 10, 11, 12, 19, 21, 22\} \\
        \hline
        4  & \{3, 5, 6, 9, 13, 16, 17, 18, 20\} \\
        \hline
        5  & \{4, 7, 8, 10, 11, 12, 15, 23, 25\} \\
        \hline
        6  & \{4, 7, 8, 10, 12, 14, 21, 23, 26\} \\
        \hline
        7  & \{3, 5, 6, 9, 13, 14, 16, 17, 18\} \\
        \hline
        8  & \{3, 5, 6, 9, 13, 14, 15, 16, 17\} \\
        \hline
        9  & \{4, 7, 8, 11, 19, 21, 24, 25, 26\} \\
        \hline
        10  & \{3, 5, 6, 11, 15, 16, 17, 18, 20\} \\
        \hline
        11  & \{3, 5, 9, 10, 13, 14, 16, 17, 18\} \\
        \hline
        12  & \{3, 5, 6, 13, 14, 15, 16, 17, 18\} \\
        \hline
        13  & \{4, 7, 8, 11, 12, 15, 16, 18, 19\} \\
        \hline
        14  & \{6, 7, 8, 11, 12, 15, 17, 18, 24\} \\
        \hline
        15  & \{5, 8, 10, 12, 13, 14, 16, 18, 22\} \\
        \hline
        16  & \{4, 7, 8, 10, 11, 12, 13, 15, 17\} \\
        \hline
        17  & \{4, 7, 8, 10, 11, 12, 14, 16, 18\} \\
        \hline
        18  & \{4, 7, 10, 11, 12, 13, 14, 15, 17\} \\
        \hline
        19  & \{3, 9, 13, 20, 22, 23, 24, 25, 26\} \\
        \hline
        20  & \{4, 10, 19, 21, 22, 23, 24, 25, 26\} \\
        \hline
        21  & \{3, 6, 9, 20, 22, 23, 24, 25, 26\} \\
        \hline
        22  & \{3, 15, 19, 20, 21, 23, 24, 25, 26\} \\
        \hline
        23  & \{5, 6, 19, 20, 21, 22, 24, 25, 26\} \\
        \hline
        24  & \{9, 14, 19, 20, 21, 22, 23, 25, 26\} \\
        \hline
        25  & \{5, 9, 19, 20, 21, 22, 23, 24, 26\} \\
        \hline
        26  & \{6, 9, 19, 20, 21, 22, 23, 24, 25\} \\
        \hline
	\end{tabular}
    \caption{Triangle-degrees and neighborhoods of graph vertices from Figure~\ref{9-reg-k3-irreg-figure}.}
    \label{tab:9reg-k3-degrees-neighbors}
\end{table}

It is worth mentioning that there exist regular $K_3$-irregular graphs that have the same order, regularity, and identical sets of $K_3$-degrees, yet, surprisingly, are not isomorphic. Such graphs appear to differ only by just a few edges. For instance, performing a single $2$-edge switch
\[
 \{9,11\}, \{10,20\} \longrightarrow \{9,10\},\{11,20\}
\]
on the graph in Figure~\ref{9-reg-k3-irreg-figure} produces a non-isomorphic graph with the same properties. In this example, even each vertex index still corresponds to the same $K_3$-degree as in the original graph.

The largest regularity for which we found an $r$-regular $K_3$-irregular graph is $r=30$. Based on our observations, such graphs seem to exist for all regularities $r \geq 9$. We halted our search at $r=30$ only due to limits of computational resources.


\section{Evolutionary search algorithm}\label{sect-algorithm}

In this section, we describe the algorithm we used to search for regular $K_3$-irregular graphs. Our method is based on an evolutionary heuristic: it maintains a population of candidate $r$-regular graphs and evolves them using mutation and selection, guided by a fitness function that promotes irregularity in triangle-degrees.

Although inspired by the general principles of evolutionary algorithms, our approach omits crossover operations due to the difficulty of merging two graphs without violating the regularity constraint. Instead, it relies on randomized mutations and elitist selection to progressively improve candidate graphs over generations.

The algorithm is implemented in C\texttt{++}, using the \textit{Boost Graph Lib}~\cite{BoostGraphLib} for graph representation and manipulation. All the discovered graphs, along with the source code, are publicly available in the Zenodo repository~\cite{ZenodoK3Irregular}.

\subsection{Main algorithm loop}

The full algorithm loop goes as follows:
\begin{enumerate}
    \item \textbf{Initialization:} Generate an initial population of $N$ graphs and evaluate their fitness.
    \item \textbf{Stopping criterion:} If a regular $K_3$-irregular graph is found, terminate the algorithm; otherwise, proceed to the next step.
    \item \textbf{Multi-offspring mutation:} For each individual in the current population, generate $m$ mutated offspring by independently applying mutation. 
    \item \textbf{Fitness evaluation:} Compute the $K_3$-degree of every vertex and evaluate the fitness of each candidate graph.
    \item \textbf{Selection:} From the enlarged pool of $mN$ candidates (optionally including some individuals from the previous generation), select the best $N$ individuals to form the next generation. The process then continues from step~2.
\end{enumerate}

\textbf{Note on population size.} Unlike classical evolutionary algorithms with fixed-size populations, our approach allows the population size to vary dynamically across phases. Each generation starts with $N$ graphs. Then, for every individual, we create $m$ copies and apply mutations independently to each of them. As a result, the mutation phase produces a temporary population of $mN$ candidates in total. In our experiments, constrained by computational resources, we typically used $N \approx 100$--$200$ and $m \approx 70$. The selection phase then reduces the expanded pool back to size $N$ by choosing the best individuals according to the fitness function. This dynamic population model enables aggressive exploration while maintaining a manageable population size across generations.

Now, we describe each part of the algorithm in more detail. The order of the following subsections is chosen to reflect the dependencies between the design of the phases (for example, the choice of the mutation strategy defines how the initial population is generated).

\subsection{Fitness function}

Initially, we designed the fitness function as a sum of two components
\[\text{fitness}(G) = f_1(G) + f_2(G).\]
Here:
\begin{itemize}
\item $f_1(G)$ measures how close the graph is to being regular --- it calculates the proportion of vertices having the same degree;
\item $f_2(G)$ measures graph $K_3$-irregularity --- this component penalizes repeated triangle-degrees by counting the number of equal pairs in the multiset of $K_3$-degrees.
\end{itemize}

However, this approach suffered from inefficiency: many mutation steps drifted the graph too far from regularity. To overcome this issue, we updated the algorithm to preserve the regularity condition under mutations and to optimize only the $K_3$-irregularity part. This not only reduced the search space, but also led to the successful discoveries of the desired graphs.

Let $M : \mathbb{N} \rightarrow \mathbb{N}_0$ be the multiplicity function of the multiset of triangle-degrees, i.e., $M(d) = |\{v \in V(G) \mid K_3\deg(v) = d\}|$.
Then the number of (unordered) pairs of vertices sharing the same triangle-degree is \[P = \sum_{\substack{d \in \mathbb{N}, \ M(d) > 1}} \binom{M(d)}{2}.\]
After this change, the fitness function is defined as follows:
\[\text{fitness}(G) = \frac{100}{P + 1}.\] This function reaches its maximum value of $100$ when all triangle-degrees are distinct, that is, when $G$ is $K_3$-irregular.

\subsection{Mutation}

Our mutation strategy evolved together with the fitness function.

\textbf{Preliminary mutations.} At the early stages, when regularity was not fixed, we used the following structural operations:
\begin{itemize}
    \item Addition of pendent paths (leaf attachments);
    \item Edge subdivision;
    \item Connecting non-adjacent vertices;
    \item Removal of vertices and edges under the constraint that the resulting graph remains connected.
\end{itemize}
However, these mutations often disrupted the degree balance, and the resulting graphs could not converge towards being regular.

\textbf{Mutations under fixed regularity.}
In order to restrict the search space to regular graphs, we used mutation operators that preserve regularity.

For even $r$, we introduced the following vertex addition scheme:
\begin{enumerate}
\item Let $G$ be a graph of even regularity $r$.
\item Let $M = \{e_1, \dots, e_{\frac{r}{2}}\}$ be a matching (in other words, $|V(M)| = r$).
\item Remove the edges in $M$ from $G$.
\item Add a new vertex $v$ and connect it to all endpoints of edges in $M$ to obtain a new graph $G'$. Formally, we have
\begin{align*}
V(G') &= V(G) \sqcup \{v\},\\
E(G') &= (E(G) \setminus M) \cup \{ vm\mid m \in V(M)\}.
\end{align*}
It is clear that the graph $G'$ has one more vertex than $G$, while keeping the same regularity $r$.
\end{enumerate}

\begin{figure}[ht]
    \begin{tikzpicture}[auto,on grid, state/.style ={circle, top color =black , bottom color = black , draw, minimum width=1mm}, inner sep=1.5pt, scale = 1]
			
			\node[state, label={[yshift=5pt]above:$v$}] (1) at (0, 4) {};
			
			\draw[line width=0.5mm] (0, 0) ellipse (3cm and 2.5cm);
			
			\node[state] (3) at (-0.5, -1.5) {};
			\node[state] (4) at (0.5, -0.5) {};
			
			\node[state] (5) at (-2, -0.25) {};
			\node[state] (6) at (-1, 1.25) {};
			
			\node[state] (7) at (2, 0.5) {};
			\node[state] (8) at (1, 1.5) {};
			
			\draw (3) -- (4);
			\draw (5) -- (6);
			\draw (7) -- (8);
			
			\node[scale=2] at (5, 0) {$\implies$};
			
			\node[state, label={[yshift=5pt]above:$v$}] (2) at (10, 4) {};
			
			\draw[line width=0.5mm] (10, 0) ellipse (3cm and 2.5cm);
			
			\node[state] (9) at (9.5, -1.5) {};
			\node[state] (10) at (10.5, -0.5) {};
			
			\node[state] (11) at (8, -0.25) {};
			\node[state] (12) at (9, 1.25) {};
			
			\node[state] (13) at (12, 0.5) {};
			\node[state] (14) at (11, 1.5) {};
			
			\draw[dashed, red] (9) -- (10);
			\draw[dashed, red] (11) -- (12);
			\draw[dashed, red] (13) -- (14);
			
			\draw (2) -- (9);
			\draw (2) -- (10);
			\draw (2) to[out=210, in=90] (11);
			\draw (2) -- (12);
			\draw (2) to[out=340, in=90] (13);
			\draw (2) -- (14);
			
    \end{tikzpicture}
    \caption{Mutation operator: adding a new vertex, even regularity (example for $r=6$).}
\end{figure}
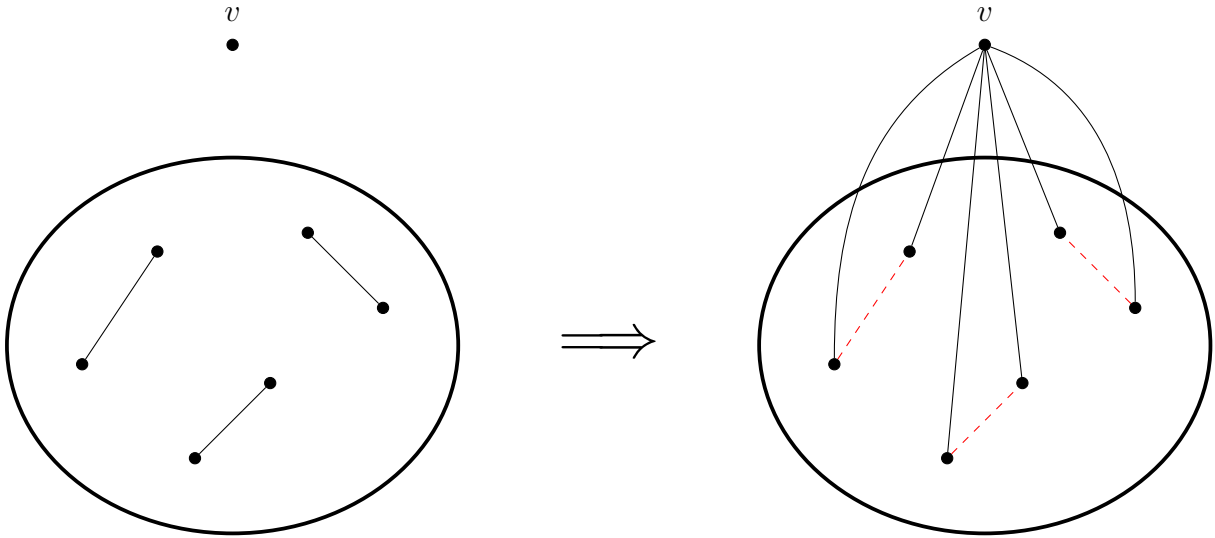

We also introduce a vertex removal operation that preserves regularity. A vertex $v$ can be removed from an $r$-regular graph $G$ if its neighborhood $N(v)$ can be partitioned into $\frac{r}{2}$ non-adjacent pairs. Thus, we remove $v$ along with all incident edges, and for each pair $\{u_i, u_j\}$ in the partitioning, we insert an edge between $u_i$ and $u_j$. The resulting graph remains $r$-regular. This mutation is particularly valuable because it enables a controlled reduction of the graph order (and consequently size) while maintaining regularity.

Consequently, the algorithm can freely move up and down in graph order, dynamically adjusting the number of vertices as needed. This flexibility helps avoid local optima and improves the exploration of the search space.

Graphs of odd regularities must have an even number of vertices, which prevents us from directly applying the even-regularity vertex addition scheme. However, a slight adaptation resolves this:

\begin{enumerate}
    \item Let $G$ be a graph of odd regularity $r$. 
    \item Add two new vertices $v$ and $w$, and connect them with an edge.
    \item Each of these vertices now requires $r-1$ additional neighbors to have degree $r$.
    \item Next, we proceed as in the even-regularity case. Select $r-1$ edges whose endpoints are all distinct, remove them, and connect each of the selected endpoints to $v$. Then, repeat the same for $w$ accordingly.
\end{enumerate}

This adjustment allows us to preserve the odd regularities under vertex addition. The removal is performed analogously: we look for two adjacent vertices $v$, $w$ such that their neighborhood (excluding the edge $\{v,w\}$) can be partitioned into non-adjacent pairs. We then remove both $v$ and $w$ and connect each pair.

\begin{figure}[ht]
    \begin{tikzpicture}[auto,on grid, state/.style ={circle, top color =black , bottom color = black , draw, minimum width=1mm}, inner sep=1.5pt, scale = 1]
			
			\node[state, label={[yshift=5pt]above:$v$}] (1) at (-1, 4) {};
			\node[state, label={[yshift=5pt]above:$w$}] (21) at (1, 4) {};
			
			\draw[line width=0.5mm] (0, 0) ellipse (3cm and 2.5cm);
			
			\node[state] (3) at (-0.5, -1.5) {};
			\node[state] (4) at (0.5, -0.5) {};
			
			\node[state] (5) at (-2, -0.25) {};
			\node[state] (6) at (-1, 1.25) {};
			
			\node[state] (7) at (2, 0.5) {};
			\node[state] (8) at (1, 1.5) {};
			
			\draw (1) -- (21);
			\draw (3) -- (4);
			\draw (5) -- (6);
			\draw (7) -- (8);

			\node[scale=2] at (5, 0) {$\implies$};
			
			\node[state, label={[yshift=5pt]above:$v$}] (2) at (9, 4) {};
			\node[state, label={[yshift=5pt]above:$w$}] (22) at (11, 4) {};
			
			\draw[line width=0.5mm] (10, 0) ellipse (3cm and 2.5cm);
			
			\node[state] (9) at (9.5, -1.5) {};
			\node[state] (10) at (10.5, -0.5) {};
			
			\node[state] (11) at (8, -0.25) {};
			\node[state] (12) at (9, 1.25) {};
			
			\node[state] (13) at (12, 0.5) {};
			\node[state] (14) at (11, 1.5) {};
			
			\draw[dashed, red] (9) -- (10);
			\draw[dashed, red] (11) -- (12);
			\draw[dashed, red] (13) -- (14);
			
			\draw (2) -- (22);
			
			\draw (2) -- (9);
			\draw (2) -- (10);
			\draw (2) to[out=210, in=90] (11);
			\draw (2) -- (12);
			\draw (2) to[out=340, in=90] (13);
			\draw (2) -- (14);
			
			\draw (22) -- (9);
			\draw (22) -- (10);
			\draw (22) to[out=210, in=80] (11);
			\draw (22) -- (12);
			\draw (22) -- (13);
			\draw (22) -- (14);
			
    \end{tikzpicture}
    \caption{Mutation operator: adding 2 new vertices, odd regularity (example for $r=7$).}
\end{figure}
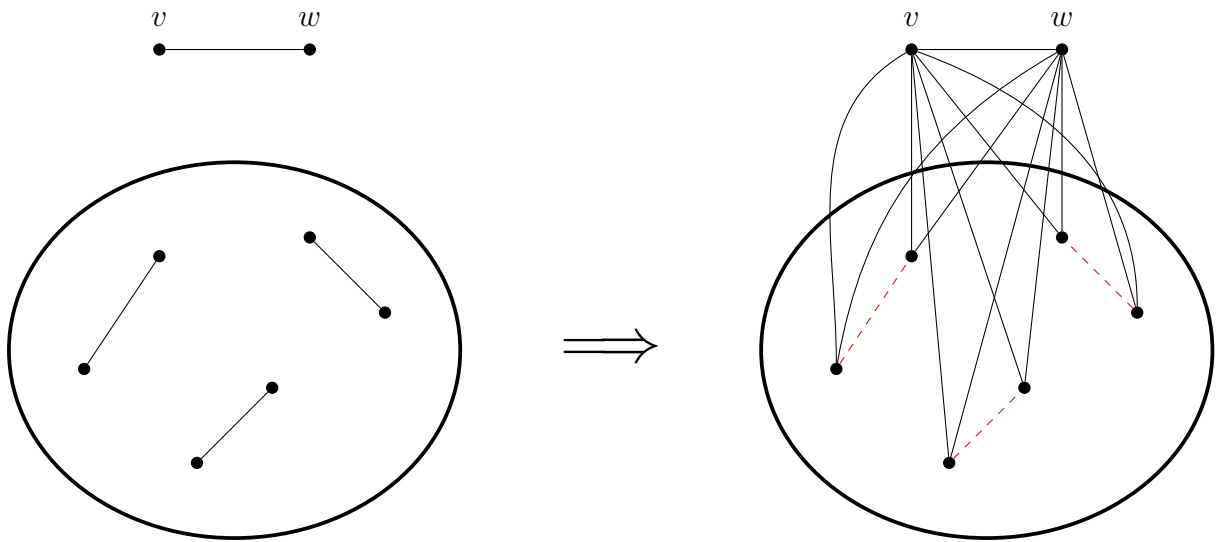

This regularity-preserving mutation proved to be effective for generating regular $K_3$-irregular graphs of regularities $r \geq 14$, but was less effective for the smaller values of $r$.

\textbf{Edge switching.}
In order to find regular $K_3$-irregular graphs for regularities in the range $9 \leq r \leq 13$, we used a local mutation based on edge switching. Initial experiments with vertex addition and removal (while preserving regularity) were used to estimate suitable values of $n$ for each fixed $r$. Once a plausible range of graph orders was identified, we fixed both $r$ and $n$ and continued using only the edge-switching mutation.

The edge-switching mutation proceeds as follows:
\begin{itemize}
    \item Randomly select two edges $u_1 v_1$ and $u_2 v_2$, such that all four vertices are distinct and that neither $u_1 u_2$ nor $v_1 v_2$ already exist in the graph.
    \item Remove the original edges and add $u_1 u_2$ and $v_1 v_2 $ instead to obtain a new graph $G'$. Formally, we have
    \begin{align*}
        V(G') &= V(G),\\
        E(G') &= (E(G) \setminus \{u_1 v_1, u_2 v_2\}) \cup \{ u_1 u_2, v_1 v_2\}.
    \end{align*}
    It is clear that the graph $G'$ has the same order and regularity as $G$.
\end{itemize}

This local mutation was sufficient to produce $K_3$-irregular graphs even for small values of $r$, starting from $r=9$.

We stress that two types of mutation should be used in combination, especially when searching for regular $K_3$-irregular graphs of large regularities. We recommend initially allowing both vertex addition/removal and edge-switching mutations. This setup enables the  algorithm to explore a feasible range of graph orders $n$ for a fixed regularity $r$, while also promoting gradual improvements in triangle-degree irregularity. Once a promising interval of $n$ values has been identified, the search continues with multiple runs using only the edge switching mutation, each performed at a fixed $n$ within that range. 

While in this work we use a static two-phase approach with separate runs, it would be natural to investigate dynamic change of mutation strategy in future versions of the algorithm.

\subsection{Initial population}

In the early experiments, we initialized the population with the complete graph $K_{r+1}$, which is trivially $r$-regular. Since the vertex adding/removing mutation preserves regularity and changes the order, this approach was sufficient.

However, when the order $n$ is also fixed (when only edge-switching mutation is used), we require a graph with the given $r$ and $n$ as the initial seed. To generate such graphs, we adopted an approach inspired by the Watts--Strogatz small-world model~\cite{watts_collective_1998}: the $n$ vertices are arranged in a ``circle'', and each each vertex is connected to its $k$ nearest neighbors. This produces a regular ring lattice. To increase structural diversity, we then apply several edge switches without calculating fitness. This produces an initial population of different `random' regular graphs of the given regularity and order.

\subsection{Selection}

We experimented with several standard selection strategies, including: Elitist Selection (Elitism), Roulette Wheel Selection (RWS), Stochastic Universal Sampling (SUS) and Tournament Selection. Among these, only Elitist Selection yielded successful results in our experiments. As previously described, the population size varies dynamically during each iteration of the algorithm. Specifically, we start with a population of size $N$, and after applying mutations, the number of individuals grows to $mN$, where $m$ is the mutation factor. The goal of the selection step is then to reduce the population size back to $N$. This is done by elitist selection of the best $N$ individuals from the pool of $mN$ candidates.

To further improve performance, we introduced one more modification: we directly carry over approximately the best $15$\% of individuals from the pre-mutation population into the pool of candidates for selection. This strategy guarantees that the best solutions found so far are preserved across generations. After that, we perform elitist selection as before, choosing the top $N$ individuals from the combined set of approximately $0.15N + mN$ candidates.

\section{Open questions}\label{sect-4}
	
	In this section, we present several open questions about regular $K_3$-irregular graphs for further research.
	
	\textbf{Question 1.} Does there exist an $8$-regular $K_3$-irregular graph? We think that the answer to this question is ``no''.
			
	\textbf{Question 2.} Is the number of triangles in a regular $K_3$-irregular graph always greater than the number of edges? All our found graphs satisfy this property.
    
    \textbf{Question 3.} Find a general (preferably, inductive) construction of $r$-regular $K_3$-irregular graphs for all $r \geq 9$. In particular, does there exist a graph operation which from a given $r$-regular $K_3$-irregular graph of order $n$ produces another $r'$-regular $K_3$-irregular graph of order $n'$ for $r'\in\{r,r+1\}$ and $n'\in\{n,n+1\}$.

    \textbf{Question 4.} For which $n,r\in\mathbb{N}$ does there exist an $r$-regular $K_3$-irregular graph of order $n$?

	\section*{Acknowledgments} 
	The authors are deeply grateful to the Ukrainian Armed Forces for keeping Leliukhivka and Kyiv safe, which gave us the opportunity to work on this paper. We also thank Vyacheslav Boyko for his valuable suggestions, which helped improve the clarity of the paper, and Vladyslav Haponenko for proofreading the manuscript at earlier stages. We would like to thank the anonymous referees for their careful reading of the manuscript and insightful comments, which helped to identify a major flaw in the proof of Theorem~\ref{8-reg-upper-bound-k3}, allowing us to correct it and improve the presentation of this paper.


\end{document}